\documentclass{article}

\usepackage{amsfonts} 
\usepackage{amsmath,hhline,latexsym}
\usepackage{graphicx}
\usepackage{amssymb}
\usepackage[latin1]{inputenc}

\usepackage{float}
\usepackage[pdftex]{lscape}
\usepackage{amsmath}
\usepackage{amsthm}
\usepackage{amsfonts,amsmath,amssymb}
\usepackage{enumerate,verbatim,color}
\usepackage{epsfig}
\usepackage{textcomp}
\usepackage{graphicx}
\usepackage{amsfonts}
\usepackage{mathrsfs}
\usepackage{upgreek}
\usepackage{mathtools}
\usepackage{pdfpages}
\usepackage{hyperref}

\usepackage{color}

\newtheorem{theorem}{\sc Theorem}[section]
\newtheorem{lemma}{\sc Lemma}[section]

\newtheorem{corollary}{Corollary}

\newcommand{\iii}{{\, \vert\kern-0.25ex\vert\kern-0.25ex\vert\, }}

\newcommand{\N}{\mathbb{N}}
\newcommand{\R}{\mathbb{R}}

\newcommand{\C}{\mathbb{C}}

\newcommand{\dd}{\partial}

\newcommand{\ra}{\rangle}
\newcommand{\la}{\langle}

\def\R{{\bf R}}

\title{Achieving energy permutation of modes in the Schr\"odinger equation with moving Dirac potentials}
\date{}
\author{Carlos Castro\thanks{M2ASAI Universidad Polit\'ecnica de Madrid, Departamento de Matem\'atica e Inform\'atica, ETSI Caminos, Canales y Puertos, 28040 Madrid, Spain. E-mail:  {\tt carlos.castro@upm.es.}} and Alessandro Duca\thanks{Universit\'e Paris-Saclay, UVSQ,
Laboratoire de Math\'ematiques de Versailles, 78035 Versailles, France. E-mail:  {\tt alessandro.duca@uvsq.fr.}}}

\begin{document}

\maketitle

\begin{abstract}
In this work, we study the Schr\"odinger equation $i\dd_t\psi=-\Delta\psi+\eta(t)\sum_{j=1}^J\delta_{x=a_j(t)}\psi$ on $L^2((0,1),\C)$ where $\eta:[0,T]\longrightarrow \R^+$ and $a_j:[0,T]\longrightarrow (0,1)$, $j=1,...,J$. We show how to permute the energy associated to different eigenmodes of the Schr\"odinger equation via suitable choice of the functions $\eta$ and $a_j$. To the purpose, we mime the control processes introduced in \cite{mio7} for a very similar equation where the Dirac potential is replaced by a smooth approximation supported in a neighborhood of $x=a(t)$. We also propose a Galerkin approximation that we prove to be convergent and illustrate the control process with some numerical simulations.  
\end{abstract}

\section{Introduction} 
\label{sec:intro}

We consider the one-dimensional Schr\"odinger equation in the interval $(0,1)$ with a measure potential supported in a moving point, 
\begin{equation} \label{eq_1}
\left\{
\begin{array}{ll}
i\partial_t \psi=-\partial^2_{xx}\psi +\eta(t) \delta_{x=a(t)}\psi, & x\in (0,1), \; t>0,\\
\psi(t,0)=\psi(t,1)=0,\\
\psi(0)=\psi^0\in L^2((0,1),\C),
\end{array} 
\right.
\end{equation}
where $\eta:[0,T]\rightarrow \R^+$ and $a:[0,T]\rightarrow (0,1)$ are sufficiently smooth functions. 

\medskip

It is well-known that slow varying potentials produce adiabatic dynamics in which the energy associated to the different modes is conserved in time. In our case, this corresponds to sufficiently small variation of both $\eta(t)$ and $a(t)$. We show that we can choose these functions $\eta(t)$ and $a(t)$ in such a way that this adiabatic regime is broken to produce a swift of the energy associated to two different modes. This can be extended to obtain a prescribed permutation of an arbitrary finite number of modes by adding several potentials supported at different moving points $a_j(t)\in(0,1)$, $j=1,...,J$.  Therefore, if we further assume that $\eta(0)=\eta(T)=0$, then the terms $\eta(t) \delta_{x=a_j(t)}$ can be interpreted as controls to achieve energy permutations of the eigenmodes for the free Schr\"odinger equation. 
We also propose a convergent numerical method to approximate the solutions of (\ref{eq_1}) in order to simulate this control.  

\medskip

The control strategy we present here is not new and it has been adapted from the works \cite{mio7, Dmitry}. In \cite{Dmitry}, the author presents how to permute eigenmodes from a spectral point of view. He considers the spectrum associated to \eqref{eq_1} by assuming $\eta$ and $a$ as real numbers. He exhibits a smart path of the parameters $\eta$ and $a$ such that, if we follow the variations of the eigenvalues of the Hamiltonian associated to \eqref{eq_1}, then a permutation is performed.

Later on, \cite{mio7} addressed this same strategy from a dynamical point of view. The authors consider the equation \eqref{eq_1} by substituting the potential $\eta(t) \delta_{x=a(t)}$ with a smooth approximation:
\begin{equation}\begin{split}\label{eq_potentiel}
i\partial_t \psi=-\partial^2_{xx}&\psi +V(x,t)\psi,\\
V(x,t)~=~ \eta(t)\,\rho^{\eta(t)}(x-a(t)&)~~~~~\text{ with }~~~\rho^\eta=\eta\,\rho(\eta\,\cdot\,)~.
\end{split}\end{equation}
Here, $\rho\in C^2(\R,\R^+)$ is a non-negative function with support $[-1,1]$ such that $\int_\R \rho(s) d s=1.$ The work shows how to control the eigenmodes of the Schr\"odinger equation by suitable functions $\eta$ and $a$. Peculiarity of \cite{mio7} is that the proposed paths for $\eta$ and $a$ are partially different from the ones considered in \cite{Dmitry} for which the dynamics of the equation \eqref{eq_potentiel} would not be well-posed in general.

Our aim is to obtain the result from \cite{mio7} by considering directly the evolution of the singular equation \eqref{eq_1}. Indeed, it is reasonable to imagine that the same control strategy holds, via the dynamics of \eqref{eq_1}, with a motion of the parameters $\eta$ and $a$ which mimes the one considered in \cite{mio7}. 

\medskip

We address two main difficulties in this work. First, the well-posedness of the Schr\"odinger equation with a Dirac potential supported at a moving point. Such kind of results was already addressed in $\R$ by \cite{Geygel} for instance, but not on bounded intervals as we have here. In this case, the quadratic form corresponding to the Hamiltonian in the Schr\"odinger equation \eqref{eq_1}:
$$
a_t(\psi)=\|\dd_x \psi\|_{L^2}^2+|\eta(t)||\psi(a(t))|^2
$$ can not be $C^2$ in time when $\psi$ is only $H^1_0$. Such hypothesis is usually required in order to apply the common well-posedness results for non-autonomous systems. Our strategy consists in introducing a tricky transformation of the interval $(0,1)$ in itself which fixes the position of the delta. The transformation modifies the equation \eqref{eq_1} in a new equivalent dynamics which is well-posed under suitable hypotheses.

Showing that the permutation of the energy associated to different eigenmodes, stated in \cite[Theorem 1.1]{mio7}, can be performed also via the dynamics of \eqref{eq_1}, with suitable parameters $a$ and $\eta$, is not a difficult issue. In fact, all the intermediate results leading to \cite[Theorem 1.1]{mio7} are still valid for the equation \eqref{eq_1} and then the same strategy can be adopted in our framework. 

Let us give an example on how the permutation considered here works. Assume that \eqref{eq_1} is a free potential Schr\"odinger equation at time $t=0$ ($\eta(0)=0$). We consider an initial data for the dynamics only given by the first two eigenmodes with different energies, i.e. 
$$
\psi^0(x)=c_1^0 \sin(\pi x)+ c_2^0 \sin(2\pi x), \quad c_1^0,c_2^0\in \mathbb{C}, \quad |c_1^0|\neq |c_2^0|. 
$$
We seek for functions $\eta(t)$ and $a(t)$ defined in the time interval $[0,T]$, with $T$ sufficiently large, such that at time $t=T$ the dynamics is again the one associated to the free potential Schr\"odinger equation ($\eta(T)=0$) and, 
$$
\psi(x,T)\sim c_1(T) \sin(\pi x)+ c_2(T) \sin(2\pi x),  
$$
with $|c_1(T)|=|c_2^0|$ and  $|c_2(T)|=|c_1^0|$. To the purpose, we divide the time interval in $5$ subintervals $[T_i,T_{i+1}]$ with $i=0,1,..,4$, $T_0=0$ and $T_5=T$. In the first $[T_0,T_1]$, we choose $a(t)=a_i$ constant, with $a_i>1/2$ and $\eta(t)$ an increasing function from $\eta(0)=0$ to $\eta_M>>1$.
In this interval, we introduce a transition between the free Schr\"odinger equation (i.e. without potential) to a new one with a large Dirac potential supported at $x=a_i$. This time interval $[T_0,T_1]$ must be sufficiently large and the dynamics sufficiently slow to guarantee an adiabatic regime. In the second subinterval $[T_1,T_2]$, we fix $\eta(t)=\eta^0$ and move $a(t)$, the support of the delta, adiabatically from $a(T_1)=a_i$ to $a(T_2)=a_i^\varepsilon=1/2+\varepsilon$ with $\varepsilon <<1$.  Until the time $T_2$, the energy of the first two modes remains the same as in the initial data. However, as $\eta_M$ is large, the solution at time $T_2$ is almost zero at $x=a_i^\varepsilon$ and the first two eigenmodes are localized in $(0,a_i^\varepsilon)$ and $(a_i^\varepsilon,1)$ respectively, i.e.
$$
\varphi_1(T_2,x)\sim\left\{ \begin{array}{ll} \sin(\frac{\pi x}{a_i^\varepsilon}),&x\in(0,a_i^\varepsilon),\\
0,& x\in(a_i^\varepsilon,1),  \end{array}\right.
\quad 
\phi_2(T_2,x)\sim\left\{ \begin{array}{ll} 0,&x\in(0,a_i^\varepsilon),\\
\sin(\frac{\pi (x-a_i^\varepsilon)}{1-a_i^\varepsilon}),& x\in(a_i^\varepsilon,1).  \end{array}\right.
$$  
In $[T_2,T_3]$, we maintain $\eta(t)=\eta_M$ constant and consider $a(t)$ as a linear function from $a_i^\varepsilon$ to $a_f^\varepsilon=1/2-\varepsilon$. This third subinterval must be sufficiently small to guarantee a continuous but non-adiabatic transition. Here, the first two eigenmodes change in such a way that now they are localized in $(a_i^\varepsilon,1)$ and $(0,a_i^\varepsilon)$ respectively, i.e.
$$
\phi_1(T_3,x)\sim\left\{ \begin{array}{ll} 0,&x\in(0,a_f),\\
\sin(\frac{\pi (x-a_f^\varepsilon)}{1-a_f^\varepsilon}),& x\in(a_f^\varepsilon,1),  \end{array}\right.
\quad 
\phi_2(T_3,x)\sim\left\{ \begin{array}{ll} \sin(\frac{\pi x}{a_f^\varepsilon}),&x\in(0,a_f^\varepsilon),\\
0,& x\in(a_f^\varepsilon,1).  \end{array}\right.
$$
Therefore, the energy associated to $\phi_1(T_2,x)$ is transmitted to $\phi_2(T_3,x)$ and the one associated to $\phi_2(T_2,x)$ pass to $\phi_1(T_3,x)$.
In the fourth interval $[T_3,T_4]$, we come back to an adiabatic regime where we change $a(t)$ from $a_f^\varepsilon$ to $a_f<1$. 
Finally, in the last interval $[T_4,T_5]$, we maintain $a(t)=a_f$ and consider $\eta(t)$ a decreasing function from $\eta_M$ to $0$. This requires again a large time interval and sufficiently slow dynamics in order to guarantee an adiabatic process. Note that the whole motion requires an adiabatic regime for slow variations of $\eta(t)$ and $a(t)$, and a continuity result for the non-adabatic transition. 
As described in \cite{mio7}, this strategy can be adapted to any permutation of energies in a finite number of modes by combining several potentials of the form $\eta(t)\delta_{x=a_i(t)}$, $i=1,...,I$ for some trajectories $a_i(t)$.

\medskip

The second difficulty we address in this work comes from the numerical approximation of the system \eqref{eq_1} containing a Dirac measure. We propose a Galerkin approximation in space with an implicit midpoint scheme in time. This Galerkin approximation involves a projection on the finite dimensional subspace generated by the first $N$  eigenfunctions of the Laplace operator in the interval $(0,1)$. As these are not eigenfunctions of the underlying operator, convergence is not straightforward and it requires a careful analysis. This method give us a simple, and easy to implement, scheme.  

However, even if the scheme is convergent when applied to  (\ref{eq_1}), the numerical simulation of the control strategy is not simple as it considers large values of $\eta(t)$ and long time simulations to guarantee the adiabatic regime. Both quantities, $\max_{t\in[0,T]}\eta(t)$ and the final time $T$, affect the error estimate. Accurate approximation of the solutions under these conditions would require a very large dimension of the finite dimensional system obtained by the Galerkin approximation and a extremely small time step. To overcome this difficulty, we observe that the proposed numerical scheme conserves the associated discrete energy and therefore it reproduces the adiabatic regime, even for large time simulations. On the other hand, the time interval where the energy swift occurs is not large and the convergence of the numerical scheme guarantees a correct simulation. Even so, some parameters as $\eta_M$, $\varepsilon$ or the lengths of the time intervals must be estimated numerically.

\medskip

\medskip

{\noindent \bf \underline{Some bibliography}}
\smallskip

The controllability via external fields of the Schr\"odinger equation has been widely studied in literature and many works addressed the problem via bilinear control fields. In other words, they considered the equation \eqref{eq_1} in presence of a bilinear potential $V(t,\cdot)=v(t) \mu(x)$ with $t\in[0,T]\subset\R^+$ instead of $\eta(t) \delta_{x=a(t)}$. Here, the real function $v$ is the control and it plays the role of the time-dependent intensity of an external field. The real function $\mu$ represents the action of the field.

The global approximate controllability of bilinear quantum systems has been proved with different techniques in the last decades. We refer to $\cite{milo,nerse2}$ for Lyapunov techniques and $\cite{ugo,nabile}$ for Lie-Galerking methods. The result was achieved via adiabatic arguments in the works $\cite{ugo2,ugo3}$.

The exact controllability of the bilinear Schr\"odinger equation is in general a more delicate matter. Indeed, it is well known that the equation is not exactly controllable in $L^2((0,1),\C)$ when $v\in L^{r}_{loc}(\R^+,\R)$ with $r>1$ and $\mu$ is sufficiently regular, even though it is well-posed. We refer to the work \cite{ball} by Ball, Mardsen and Slemrod for further details on these two results.

The turning point for this kind of studies was the idea, introduced by Beauchard in $\cite{be1}$, of controlling the equation in suitable subspaces of $L^2((0,1),\C)$. Following this approach, different works addressed the controllability issue on bounded interval as \cite{laurent,mio2,mio1,morgane1,morganerse2}. The problem was also studied on quantum graphs in \cite{mio5,mio6,mio3,mio4} and on the two dimensional disc in \cite{Moyano}.

\medskip

A peculiarity of the control proposed in our work and in \cite{mio7} is the coupling between adiabatic and non-adiabatic motions, which is very unusual for the common results of this type. In addition, the existing techniques for the bilinear controls can not be directly applied for the Schr\"odinger equation in presence of delta-potentials as in \eqref{eq_1}. From this perspective, we provide a new and different way to permute eigenmodes of quantum systems via external control fields.

Another peculiarity of our techniques is the simplicity of the controls. Once the paths for the delta-potentials are designed, it is enough to sufficiently slow down the adiabatic parts of the dynamics and to accelerate the transitions in order to obtain the permutations. We refer to Section \ref{numerical} for further details on how to proceed in some explicit examples.

\medskip

{\noindent \bf \underline{Scheme of the work}}

\smallskip

In Section \ref{sec_well}, we address the well-posedness of \eqref{eq_1} stated in Theorem \ref{well_general}. In Section \ref{permutation}, we show how to adapt the permutations of modes described in \cite{mio7} to our setting. This is in fact straightforward and we only underline the main ideas of the proof by presenting those auxiliary results from \cite{mio7} adapted to our context.
In Section \ref{galerkin}, we introduce the numerical approximation and prove its convergence.
In Section \ref{numerical}, we discuss some numerical simulations where we illustrate the control strategy described above. Finally, in Section \ref{comments} we give some final comments and difficulties to capture other theoretical results in the simulations.

\section{Well-posedness}\label{sec_well}
The aim of this section is to ensure the existence and the unicity of solutions of the the Schr\"odinger equation \eqref{eq_1}. We restrict ourselves to the case of one single Dirac potential for simplicity but the results hold also when considering a finite number of nonintersecting Dirac potentials. 

 To the purpose, we consider the following equivalent system in $L^2((0,1),\C)$
\begin{equation} \label{eq_1_bis}
\left\{
\begin{array}{ll}
i\partial_t \psi=-\partial^2_{xx}\psi, & x\in (0,a(t))\cup(a(t),1), \ \ t>0,\\
\psi(t,0)=\psi(t,1)=0,\\
\psi(t,a(t)^-)=\psi(t,a(t)^+),\\
\dd_x\psi(t,a(t)^+)-\dd_x\psi(t,a(t)^-)=\eta(t) \psi(t,a(t)),\\
\psi(0)=\psi^0\in L^2((0,1),\C).\\
\end{array} 
\right.
\end{equation}
An abstract setting for the dynamics of \eqref{eq_1_bis} is given by the following equation in $L^2((0,1),\C)$ 
$$i\partial_t \psi=A_t\psi,\ \ \ \ \ \  \ \ \  \ \ \ \ \ \  \ \ \  A_t=-\partial^2_{xx},$$
$$D(A_t)=\big\{\psi\in H^2((0,a(t))\cup(a(t),1))\cap H^1_0(0,1)\ :\ \dd_x\psi(a(t)^+)-\dd_x\psi(a(t)^-)=\eta(t)
\psi(a(t))\big\}.$$
Note that $A_t$ is self-adjoint and positive definite. The main result of this section is the following.

\begin{theorem}\label{well_general}
Let $a\in C^3([0,T],(0,1))$ and $\eta\in C^2([0,T],\R^+)$.
Equation \eqref{eq_1_bis} generates a unitary flow in $L^2(0,1)$
 and for any $\psi_0\in D(A_0)$, the corresponding solution of the equation \eqref{eq_1_bis} is
$C^0([0,T];D(A_t))\cap 
C^1([0,T];L^2)$. Finally, for any $\psi_0\in H^1_0$, the equation \eqref{eq_1_bis} admits a solution in
$C^0([0,T];H^1_0)\cap 
C^1([0,T];H^{-1})$.
\end{theorem}
\begin{proof}Let us start by considering $\eta>0$ as a constant function. We discuss how to generalize the result in the final part of the proof.

\noindent
{\bf 1) Preliminairies. } The proof is based on the following idea. Let $I_t=[a(t)-\epsilon,a(t)+\epsilon]$ be a family 
of intervals such that $0<\epsilon<1/2$ is so that $I_t\subset(0,1)$ for every $t$.  We define a family smooth diffeomorphisms $h(t,x)$ which 
is $C^3$ in both the variables $t$ and $x$, so 
that, for every $t\in[0,T]$,
\begin{equation}\begin{split}\label{propiet}
h(t,(0,1))=(0,1),\ \ \ \ \ \ \ \ \ \ \ \ \ h(t,0)=0,\ \ \ \ \ &\ \ \ \ \ \ \ h(t,a(t))=1/2,\ \ \ \ \ \ \ \ \ \ \ \ h(t,1)=1,\\
h\big(t,(0,a(t)-\epsilon)\big)=\Big(0,\frac{1}{2}-\epsilon\Big),  \ \ \ h(t,I_t)=\Big[\frac{1}{2}&-\epsilon,\frac{1}{2}+\epsilon\Big]\  \ \ \ h\big(t,(a(t)+\epsilon,1)\big)=\Big(\frac{1}{2}+\epsilon,1\Big).                                                      \end{split}\end{equation}
Our purpose is to apply the transformation $h$ to the interval $(0,1)$ in order to fix the position of the delta by translating the interval $I_t$ in $\big [\frac{1}{2}-\epsilon,\frac{1}{2}+\epsilon\big]$. A possible choice is to set, for $x\in I_t$, $$h(t,x)=x-a(t) + 1/2.$$ We investigate how to define $h$ outside $I_t$ in order to guarantee \eqref{propiet}. For $x\in (0,a(t))$, we take
\begin{equation*}\begin{split}
h(t,x)=\frac{a(t)-1/2}{(a(t)-\epsilon)^4}  x^4 -  4\frac{a(t)-1/2}{(a(t)-\epsilon)^3}  x^3 + 6\frac{a(t)-1/2}{(a(t)-\epsilon)^2}  x^2 - 4\frac{a(t)-1/2}{a(t)-\epsilon} x+x.
\end{split}\end{equation*}
This is an invertible polynomial in $(0,a(t)-\epsilon)$ such that the continuity of the first three order derivatives in space and in time of $h$ is guaranteed at the point $x=a(t)-\epsilon$. In addition, we have that $h(t,0)=0$ and $h(t,a(t)-\epsilon)=1/2-\epsilon.$
Similarly for $x\in (a(t)+\epsilon,1),$ we can take
\begin{equation*}\begin{split}
                  h(t,x)=&\frac{a(t)-1/2}{(1-a(t)-\epsilon)^4}  (1-x)^4 -  4\frac{a(t)-1/2}{(1-a(t)-\epsilon)^3}(1-  x)^3 + 6\frac{a(t)-1/2}{(1-a(t)-\epsilon)^2} (1- x)^2 \\
                  &- 4\frac{a(t)-1/2}{1-a(t)-\epsilon}(1- x)+x.
                  \end{split}\end{equation*}
                  Clearly, these are not the only possible definitions for the diffeomorphisms $h$ and their expressions are not relevant for the proof. However, we explicit them in order to convince the reader of the existence of such a familiy of diffeomorphisms $h$ verifying the relations \eqref{propiet}.

                  \smallskip

\noindent
{\bf 2) Fixing the domain. }As introduced before, our aim is to fix the position of the delta by using the diffeomorphisms $h$. This transformation modifies the Schr\"odinger equation \eqref{eq_1_bis} as presented in the work \cite{mio9}.  To provide the new formulation of the equation, we introduce the pullback operator  
\begin{equation}\label{pullback}
h^*(t)~:~\psi\in L^2((0,1),\C)~\longmapsto~\psi\circ h=\phi(h(t,\cdot)) \in L^2((0,1),\C)
\end{equation}
and its inverse, the pushforward operator, defined by 
\begin{equation}\label{pushforward}
h_*(t)~:~\phi\in L^2((0,1),\C)~\longmapsto~\psi\circ h^{-1}=\phi(h^{-1}(t,\cdot)) \in L^2((0,1),\C)~.
\end{equation}
By following the theory developed in \cite{mio9}, we denote
\begin{equation}\label{pullback_sharp}
h^\sharp(t)~:~\psi\in L^2((0,1),\C)~\longmapsto~\sqrt{|\dd_x h(t,\cdot)|}\,(\psi\circ h)(t) \in L^2((0,1),\C).
\end{equation}
We also call $h_\sharp(t)$ its inverse 
\begin{equation}\label{pushforward_sharp}
h_\sharp(t)=(h^\sharp(t))^{-1}~:~\phi\mapsto (\phi / \sqrt{|\dd_x h(t,\cdot)|})\circ h^{-1} \in L^2((0,1),\C).
\end{equation}
Notice that the relation $\|h^\sharp(t)u\|_{L^2}=\|u\|_{L^2}$ yields that $h^\sharp(t)$ and 
$h_\sharp(t)$ are isometries and they preserve the Hamiltonian structure of the Schr\"odinger equation through the 
change of variables.  From now on, we omit the time dependence from $h$, $h^\sharp$ and $h_\sharp$ when it is 
not necessary.

We use the family of diffeomorphisms $h$ in 
order to act a change of variable in $(0,1)$ and we use $h^\sharp$ in order to rewrite \eqref{eq_1_bis} in a 
equivalent equation in $L^2(0,1)$ where the internal 
dynamical boundary conditions are now placed in the point $1/2$. In particular,
$$h^\sharp (D(A_t))=\{\phi\in H^2((0,1/2)\cup(1/2,1))\cap H^1_0(0,1)\ :\ 
\dd_x\psi(1/2^+)-\dd_x\psi(1/2^-)=\eta 
\psi(1/2)\}.$$
The jump conditions is only displaced at the point $1/2$ and it does not change after the the transformation 
$h^\sharp$. Indeed, $h$ translates each $I_t$ in $(1/2-\epsilon,1/2+\epsilon)$ and fixes the position $a(t)$ in $1/2$, while $\dd_x h(t,\cdot)$ is 
constantly equal to $1$ in a neighborhood of $1/2$.
If the equation obtained after the transformation $h^\sharp$ admits a unitary propagator $U_t$, then $h_\sharp U_t 
h^\sharp$ is the unitary 
propagator associated to the dynamics of \eqref{eq_1}.

\smallskip

\noindent
{\bf 3) Existence and unicity of solutions of the transformed dynamics.} When $\psi$  is solution of \eqref{eq_1} (at least from a formal sense), $\phi:=h^\sharp \psi$ solves the following equation (see [identity\ (1.8);\ 2] for further details):
\begin{equation}\label{eq_transf}i\partial_t 
\phi(t)=h^\sharp H(t)h_\sharp \phi(t),\ \ \ \  \text{in}\ \ \ \ (0,1)\ \ \ \ \text{with}\end{equation}
$$H(t):=-\Big[\big(\dd_x+i{M_h}\big)\circ\big(\dd_x+iM_h\big)+M_h^2\Big],\ 
\ 
\ 
\ \ \ ~~~M_h(t,x)=-\frac 12 (h_*\partial_t h)(t,x).$$
We denote by $D(h^\sharp H(t)h_\sharp)$ the space $h^\sharp D(A_t)$. We notice that studying $\la\phi_1,h^\sharp H(t)h_\sharp\phi_2 \ra_{L^2}$ for 
every $\phi_1,\phi_2\in h^\sharp D(A_t)$ is equivalent to study $\la\psi_1,H(t)\psi_2 
\ra_{L^2}$ for 
every $\psi_1,\psi_2\in D(A_t)$ and
$$\la\psi_1,H(t)\psi_2 
\ra_{L^2}=\Big\la\psi_1,-\dd_{xx}^2\psi_2 +\frac{i}{2}\dd_x\big(h_*(\dd_t h)\psi_2 ) 
+\frac{i}{2}h_*(\dd_t 
h)\dd_x \psi_2 
\Big\ra_{L^2}.$$
Now, $A_t=-\dd_{xx}^2$ is self-adjoint in $D(A_t)$ and the operator $\frac{i}{2}\dd_x\big(h_*(\dd_t 
h)\cdot) 
+\frac{i}{2}h_*(\dd_t 
h)\dd_x$ is symmetric in such space. As a consequence, $h^\sharp H(t)h_\sharp$ is a family of self-adjoint 
operators in $D(h^\sharp H(t)h_\sharp)$.

%
 In order to ensure the well-posedness of the equation \eqref{eq_transf}, we would like to apply \cite[Theorem A.1]{mio9} which rephrases the theory from the 
work of Kisynsky \cite{Kisy}. For every $\psi\in D(h^\sharp H(t)h_\sharp)$,
\begin{equation}\label{dissipative}\begin{split}q(\psi):=&\la\psi,h^\sharp H(t)h_\sharp\psi 
\ra_{L^2}=\la h_\sharp\psi, H(t)h_\sharp\psi 
\ra_{L^2}=\|\big(\dd_x+i{M_h}\big)(h_\sharp \psi)\|_{L^2}^2+|\eta||(h_\sharp \psi)(a(t))|^2\\
&- \|{M_h}(h_\sharp \psi)\|_{L^2}^2
=\|h^\sharp\big(\dd_x+i{M_h}\big)h_\sharp \psi\|_{L^2}^2+|\eta||\psi(1/2)|^2- \|h^\sharp{M_h}h_\sharp \psi\|_{L^2}^2.\end{split}\end{equation}
The Friedrichs extension of the quadratic form $q(\cdot)$ defined on $C_0^\infty ((0,1),\C)$ is the quadratic form $q(\cdot)$ defined in $H^1_0((0,1),\C)$ and $D(|h^\sharp H(t)h_\sharp|^\frac{1}{2})=H^1_0((0,1),\C)$. The sesquilinear form $\Phi_t(\phi,\psi)=\big\la\phi,h^\sharp H(t)h_\sharp\psi\big\ra$ from $H^1_0$ to $H^{-1}$ has the same regularity of $M_h$ (and then of $\dd_th $), which is $C^2$ in time as soon as $a$ is $C^3$. Finally, there exist $\gamma>0$ and $\kappa\in\R$, depending only on $h$, such that
$$\Phi_t(\psi,\psi)\geq\gamma \|\dd_x \psi\|_{L^2}^2-\kappa\|\psi\|_{L^2},\ \ \ \ \  \ \ \ \ \ \forall \psi\in D(h^\sharp H(t)h_\sharp).$$
The statement is ensured thanks to the results of Kisynsky \cite{Kisy} stated in \cite[Theorem A.1]{mio9}.

\smallskip

\noindent
{\bf 4) Conclusion.} When $\eta$ is not just a constant functions, the well-posedness follows as above. Indeed, if we define the sesquilinear form $\Phi_t(\phi,\psi)$ corresponding to the equation (it is the same of \eqref{dissipative} with $\eta$ depending on time), then we notice that it is $C^2$  as soon as $a$ is $C^3$ and $\eta$ is $C^2$.
\end{proof}

We are finally ready to ensure the well-posedness of the equation \eqref{eq_1} in presence of more moving deltas. To the purpose, we introduce the corresponding abstract setting given by the following equation in $L^2((0,1),\C)$ 
\begin{equation}\label{moredeltas}i\partial_t \psi=\tilde A_t\psi,\end{equation}
where the operator $\tilde A_t=-\partial^2_{xx}$ is defined on the domain
\begin{align*}D(\tilde A_t)=\Big\{&\psi\in H^2\big((0,a_1(t))\cup(a_1(t),a_2(t))\cup...\cup(a_{J-1}(t),a_J(t))\cup(a_J(t),1)\big)\cap H^1_0(0,1)\ :\\
&\dd_x\psi(a_j(t)^+)-\dd_x\psi(a_j(t)^-)=\eta(t)
\psi(a_j(t)),\ \ \ \ \forall j\leq J\Big\}.\end{align*}

\begin{corollary}\label{well_general_coro}
Let $\eta\in C^2([0,T],\R^+)$ and $\{a_j\}_{j\leq J}\subset C^3([0,T],(0,1))$ be a familiy of function which never intersect.
Equation \eqref{moredeltas} generates a unitary flow in $L^2(0,1)$
 and for any $\psi_0\in D(\tilde A_0)$, the corresponding solution is
$C^0([0,T];D(\tilde A_t))\cap 
C^1([0,T];L^2)$. Finally, for any $\psi_0\in H^1_0$, the solution belong to
$C^0([0,T];H^1_0)\cap 
C^1([0,T];H^{-1})$.
\end{corollary}
\begin{proof}
Since the functions $a_j$ never intersect, there exists a smooth diffeomorphisms $h$ from $(0,1)$ to himself which fixes the positions of all the delta-potentials at the same time. The result is proved by using techniques leading to Theorem \ref{th} by considering this diffeomorphism $h$ .
\end{proof}

\section{Permutation of the eigenmodes}\label{permutation}

In this section, we prove the following result that states that we can permute a finite number of eigenmodes of the Dirichlet Laplacian on $(0,1)$ by a quasi-adiabatic motion of several delta-potentials in $\eqref{eq_1}$, i.e. a potential of the form $\eta(t)\sum_{j=1}^J \delta_{x=a_j(t)}$, for some $J\geq 1$. 

\begin{theorem}\label{th}
Consider $M\in \mathbb{N}^*$ and let $\sigma:\{ 1,...,M\}\to \{ 1,...,M\}$ be a permutation affecting the first $M$ integers. There exist $T>0$ (sufficiently large), $\varepsilon,\kappa>0$  and
\begin{itemize}
\item $\eta\in C^\infty([0,T],\R^+)$ with $\|\eta'\|_{L^\infty([0,T],\R)}\leq \kappa$ and $\eta(0)=\eta(T)=0$,
\item $a_j\in C^\infty ([0,T],(0,1))$ with $\|a'\|_{L^\infty([0,T],\R)}\leq \kappa$, for $j=1,...,M-1$, 
\end{itemize}
such that the evolution defined by the linear Schr\"odinger equation \eqref{eq_1}, in presence of the potential $\eta(t)\sum_{j=1}^J \delta_{x=a_j(t)}$, realizes the quasi-adiabatic 
permutation $\sigma$. Namely, let $\Gamma_{s}^{t}$ be the unitary propagator generated in the time 
interval $[s,t]\subset[0,T]$ by such equation. For all 
$k\leq M$, there exist $\alpha_k\in\C$ with 
$|\alpha_k|=1$ such that 
$$\big\|~\Gamma_{0}^{T}\sin(k\pi x)\,-\,\alpha_k\sin(\sigma(k)\pi x)~\big\|_{L^2}\leq \varepsilon~. $$
\end{theorem}

Theorem \ref{th} corresponds to \cite[Proposition 6.1]{mio7} when we consider the equation \eqref{eq_1} instead of \cite[equation\ (SE)]{mio7} with \cite[potential\ (6.3)]{mio7}. Its proof follows by applying the same method leading to \cite[Proposition 6.1]{mio7} and it combines three main ideas: 
\begin{enumerate}
\item the adiabatic process when $\eta(t)$ and $a(t)$ change sufficiently slow;
\item a continuity result: for a sufficiently short time interval $[t_1,t_2]$, a constant value for $\eta(t)=\eta_M$, and some suitable functions $a_j(t)$, $j=1,...,M-1$, the spatial distribution of the energy of the solutions remains almost the same in $[t_1,t_2]$. 
\item the strategy to choose $\eta(t)$ and $a_j(t)$, $j=1,...,M-1$, that performs the prescribed permutation.  
\end{enumerate}
 
We refer to \cite{mio7} for a detailed description of the process. Here, we state the key results related with the first two previous ideas (reformulated in the context of equation \eqref{eq_1}) and the strategy that we follow in the numerical experiments below. Following \cite{mio7},  the completion of the proof with these results is straightforward.

\subsection{The adiabatic regime}

We start with the adiabatic result. Let $a\in C^3((0,T),(0,1))$ and $\eta\in C^2((0,T),\R^+)$.
It is easy to check that for every $t\in (0,T)$ the spectrum of the operator $H(t)+\eta(t) \delta_{x=a(t)}$ is composed by 
simple and isolated eigenvalues. We denote by $\lambda(t)$ one of them, $\phi(t)$ is the corresponding normalized eigenfunction 
and $P(t)$ is associated to a spectral projector. Now, $\lambda(t)$, $\phi(t)$ and $P(t)$ continuously depend on 
$t$. Following the classical adiabatic principle, we expect that a sufficiently slow dynamics of \eqref{eq_1} 
starting from a quantum state close to $\phi(0)$ stay close to $\phi(t)$ up to a phase shift. The slowness of the 
dynamics is represented by a parameter $\epsilon>0$ which is considered between the times $t=0$ and 
$t=T/\epsilon$. We rewrite \eqref{eq_1} in the following Schr\"odinger equation
\begin{equation} \label{eq_1_adiabatic}
\left\{
\begin{array}{ll}
i\partial_t \psi_\epsilon=-\partial^2_{xx}\psi_\epsilon +\eta(\epsilon t) \delta_{x=a(\epsilon t)}\psi_\epsilon, & 
x\in (0,1), \; t\in (0,T/\epsilon),\\
\psi_\epsilon(t,0)=\psi_\epsilon(t,1)=0, & t>0,\\
\psi_\epsilon(0,x)=\psi^0(x), & x\in(0,1),
\end{array} 
\right.
\end{equation}

\begin{theorem}\label{th_adiabatic}
Consider the above framework. The following convergence is uniform in $t\in (0,T)$
$$\la P(t)\psi_\epsilon(t/\epsilon)|\psi_\epsilon(t/\epsilon)\ra_{L^2} ~~\xrightarrow[~~\epsilon\longrightarrow 
0~~]{} ~~\la P(0)\psi^0|\psi^0\ra_{L^2}~.$$
\end{theorem}
\begin{proof}
 Due to Theorem \ref{well_general}, we know how to define solutions of \eqref{eq_1_adiabatic} by rewriting the system in the form of \eqref{eq_1_bis}. Now, if we apply the family of diffeomorphisms $h$ introduced in the proof of Theorem 
\ref{well_general}, then we obtain a family of self-adjoint operators associated to a family of quadratic forms smooth in time (as showed in the proof of the mentioned theorem). In this framework, the results of Nenciu \cite[Remarks;\ p.\ 16;\ (4)]{Nen} are valid (see also Teufel \cite[Theorem\ 4.15]{Teufel} or the works \cite{ASY, Ga, JoPf}) and the adiabatic theorem is ensured such as in \cite[Section 5]{mio9}.\qedhere
\end{proof}

\subsection{Continuity result}\label{section_fast_translation} 

In this subsection, we state the second main ingredient to prove Theorem \ref{th}, which corresponds to \cite[Lemma 4.1]{mio7}. 

\begin{lemma}\label{lemma-compare}
Let $\psi(t)$ be the solution of the Schr\"odinger 
equation \eqref{eq_1} in a time interval $ [t_1,t_2]$ with $L^2((0,1),\C)$ initial data and $\eta>0$ constant. 
Take 
any function $f\in C^1([t_1,t_2],H^2((0,1),\C)\cap H^1_0((0,1),\C))$
such that $f(t,x)$ vanishes in $a(t)$ for every $t\in[t_1,t_2]$. Then, for all $t\in [t_1,t_2]$,
\begin{equation}\label{eq-bootstrap-intermediate}
\|\psi(t)-f(t)\|_{L^2}^2 \leq \|\psi(t_1)-f(t_1)\|_{L^2}^2 + C(t_2-t_1),
\end{equation}
with $C$ independent of the choice of functions $\eta$ and $a$, and given by
$$C=2\sup_{t\in [t_1,t_2]} \big(\,\|\psi(t_1)\|_{L^2} \|\partial_{xx}^2 f(t)\|_{L^2} + 
(\|f(t)\|_{L^2}+\|\psi(t_1)\|_{L^2})\|\partial_t f(t)\|_{L^2}\,\big).$$
\end{lemma}
\begin{proof}
The proof follows as the one of \cite[Lemma 4.1]{mio7} .\qedhere
\end{proof}

\subsection{The control strategy} \label{sec_cont_stra}

As described in the introduction for the permutation of the first two modes, the general process has several steps in different consecutive time intervals. The number of intervals will depend on the complexity of the permutation. We refer to the numerical examples below for further examples. In the first step, we increase adiabatically the $M-1$ Diracs by growing the value of $\eta$ from $0$ to a large $\eta_M=\eta(T_1)$. According to Theorem \ref{th_adiabatic}, this is possible as long as we consider a slow variation of $\eta$ and therefore a very long time. The supports of the Diracs are fixed in this time interval and we call them $$0=b_0^s<b_1^s<...<b_{M-1}^s<1=b_M^s.$$ We choose this values in such a way that the lengths of the subientervals $I_j^s =[b_j^s,b_{j+1}^s]$,  given by $l_j^s=b_{j+1}^s-b_j^s$, $j=0,...,N-1$, are in decreasing order but all of them are larger than $l_1^s/2$, i.e.
$$
0>l_1^s>l_2^s>...>l_{M-1}^s>l_1^s/2.
$$
After this time step $[0,T_1]$, the dynamics will be close to the one associated to a split domain or a domain with $M-1$ internal zero Dirichlet boundary conditions at $x=b_j^s$, $j=1,...,M-1$. In this limit problem, the eigenvalues $\lambda_j(T_1)$ are the union of the eigenvalues of each one of the interval $I_j$ and, by our choice, they are ordered in such a way that
$\lambda_j(T_1)$ is the first eigenvalue of the Dirichlet Laplace operator in $ I_j^s$ when  $j=0,...,M-1$. 

So far the process has been adiabatic and the $j-$mode will be almost the first mode of the interval $I_j^s$, extended by zero to $x\in(0,1)$, when  $j=0,...,M-1$. In particular, the energy associated to the $j$-mode will be concentrated in $I_j^s$. 

In the next time interval $[T_1,T_2]$, we fix the function $\eta(t)=\eta_M$ and move the support of the Diracs from $a_j(T_1)=b_j^s$ to $a_j(T_2)=b_j^f$ in such a way that the length of the subintervals $I_j^f =[b_j^f,b_{j+1}^f]$,  given by $l_j^f=b_{j+1}^f-b_j^f$, $j=0,...,N-1$, are reordered according to the prescribed permutation, i.e.
$$
0>l_{\sigma(1)}^f>l_{\sigma(2)}^f>...>l_{\sigma(M-1)}^f>l_{\sigma(1)}^f/2.
$$ 
The movement must be adiabatic, and therefore $|a_j'|$ must be small, up to some times $t_1<t_2<...<t_N$ where fast non-adiabatic transitions are required. The transitions occur at a time $t_r$ when two different intervals $I_j(t_r)=[a_j(t_r), a_{j+1}(t_r)]$ and $I_k(t_r)=[a_k(t_r), a_{k+1}(t_r)]$ have the same length $l_{j}(t_r)=l_k(t_r)$. According to Lemma \ref{lemma-compare}, the distribution of the energy between the different subintervals is maintained before and after the transition time $t=t_r$.   

The process in $[T_1,T_2]$ requires a detailed analysis of the trajectories $a_j(t)$ of the Diracs to capture all the transitions. In practice, we move sequentially each one of the Diracs in different time subintervals. In this way, transitions are easier to capture. The number of subintervals where we have to change from adiabatic to transition depends on the complexity of the permutation.   

In the final step $[T_2,T]$, the supports of the Diracs are fixed again and we  consider $\eta(t)$ slowly decreasing from $\eta(T_2)=\eta_M$ to $\eta(T)=0$. The adiabatic regime ensures that the permutation introduced in the previous time interval is conserved.

\section{Numerical approximation}\label{galerkin}

In this section, we propose a numerical method to approximate the dynamics of the Schr\"odinger equation \eqref{eq_1_bis} and in particular to simulate the permutation of the energy modes described above. We divide the section in four more subsections where we introduce a Galerkin approximation in the space variable, prove the convergence of the method, write an equivalent matrix formulation and propose a second order time scheme for the approximation of the semidiscrete problem.

\subsection{Galerkin approximation}

We consider the operator $A_t$ defined in Section \ref{sec_well} and its eigenpairs denoted by $(\lambda_k(t),\phi_k(x,t))$ where  $\{ \phi_k(x,t)\}_{k\in\N^*}$ is a Hilbert basis of $L^2(0,1)$ for every $t>0$. We call $$\{ w_k(x)\}_{k\in \mathbb{N}^*}$$ another Hilbert basis of $L^2(0,1)$ composed by some eigenfunctions of the Dirichlet Laplacian in $(0,1)$ and $\{\nu_k\}_{k\in \mathbb{N}^*} $ the associated eigenvalues. In our case, 
$$
w_k(x)=\sqrt{2}\sin(k\pi x), \qquad \nu_k=k^2\pi^2, \quad  k\geq 1.
$$
Consider $X_N\subset X$ the finite dimensional space generated by $\{ w_k\}_{k=1}^N$ and the usual projection 
$$
P^N:X\to X_N.
$$ 
We define the finite dimensional approximation of \eqref{eq_1_bis} by 
\begin{equation} \label{eq_ab1N}
i\partial_t \psi_N = P^N A_t \psi_N, \quad \psi_N(0)=P^N \psi^0, \quad \psi_N(t)\in X_N.
\end{equation}
Now, $P^NA_t$ is self-adjoint and the finite dimensional system \eqref{eq_ab1N} can be written in matrix form as a system of ODE with continuous coefficients (see Section \ref{matrixform}). Therefore, existence and unicity of solutions holds, and the $L^2-$norm is conserved, i.e.
$$
\| \psi_N(t)\|_{L^2} = \| \psi_N(0)\|_{L^2}, \quad t>0. 
$$

\subsection{Convergence of the Galerkin approximation}

The aim of this section is to prove that the Galerkin approximation converges to the solution of the continuous problem \eqref{eq_1_bis}. We have the following result. 

\begin{theorem} \label{th_conv}
Assume that $a$ and $\eta$ satisfy the hypotheses of Theorem \ref{well_general} to guarantee the existence of a solution $\psi\in C([0,T];H^1_0)$ of \eqref{eq_1_bis} with initial data $\psi^0\in H^1_0$. Let $\psi_N$ be the solution of the corresponding finite dimensional approximation \eqref{eq_ab1N}. 
Then, 
\begin{equation} \label{eq_th1}
\|\psi(t)-\psi_N(t)\|_{L^2} \leq \left(1 +  2T\frac{\eta_M}{\pi}\right) \frac{\sqrt{\eta_M}}{\sqrt{3} \sqrt{N}} \| \psi( t)\|_{L^\infty((0,T);H^1_0)}, 
\end{equation}
where $\eta_M=\max_{t\in[0,T]}\eta(t)$.
\end{theorem}
\begin{proof}
We write the solution $\psi$ as linear combination of eigenfunctions,
\begin{equation}
\psi(t,x)=\sum_{k=1}^\infty \hat \psi_k (t) \phi_k(t,x), \quad\quad \hat \psi_k (t) = \int_0^1 \psi(t,x) \overline{\phi_k(t,x)} \; dx .
\end{equation}
Note that 
$$
\|\psi(t)\|^2_{D(|A_t|^n)} =  \sum_{k=1}^\infty \lambda_k^{2n}(t) |\hat \psi_k (t)|^2, \quad n=0, \; 1/2, \; 1, 
$$
and in particular $D(|A_t|^{1/2})=H^1_0(0,1)$.  Let us consider the following two estimates ensured by Lemma \ref{le_1} and Lemma \ref{le_2} below:
\begin{equation*}
\| \psi(t)-\psi_N(t)\|_{L^2}\leq \left(1 +  2T\frac{\eta_M}{\pi}\right) \; \| (I-P^N)\psi\|_{L^\infty((0,T);H^1_0)}, 
\end{equation*}
\begin{equation*} 
\| (I-P^N) \phi_k(t,\cdot) \|_{H^1_0} \leq \frac{\sqrt{2}}{\pi}\frac{\sqrt{\eta(t)}}{\sqrt{N}}\sqrt{\lambda_k-\nu_1},\ \ \ \ \  \ \forall k\in\N^*.
\end{equation*}
The proof follows by gathering these inequalities in the following computations 
\begin{equation*}\begin{split}
& \| \psi(t)-\psi_N(t)\|_{L^2}\leq \left(1 +  2T\frac{\eta_M}{\pi}\right) \left\| (I-P^N)\sum_{k=1}^\infty \hat \psi_k(t)\phi_k(t,x)\right\|_{L^\infty((0,T);H^1_0)}\\ \nonumber
& \quad \leq \left(1 +  2T\frac{\eta_M}{\pi}\right)\max_{t\in[0,T]} \sum_{k=1}^\infty |\hat \psi_k(t)| \; \| (I-P^N)\phi_k(t,\cdot)\|_{H^1_0}  \\ \nonumber
& \quad \leq \left(1 +  2T\frac{\eta_M}{\pi}\right) \max_{t\in[0,T]}\left(\sum_{k=1}^\infty |\hat \psi_k(t)|^2 \; \lambda_k(t)\right)^{1/2} \;  \left(\sum_{k=1}^\infty  \frac{\| (I-P^N)\phi_k(t,\cdot)\|_{H^1_0}^2}{\lambda_k(t)}\right)^{1/2}   \\ \nonumber
& \quad \leq \left(1 +  2T\frac{\eta_M}{\pi}\right)\| \psi(t) \|_{L^\infty((0,T);H^1_0)} \; \max_{t\in[0,T]} \left(\sum_{k=1}^\infty  \frac{2\eta(t)(\lambda_k(t)-\nu_1)}{\pi^2N\lambda_k(t)}\right)^{1/2} \\ 
& \quad \leq \left(1 +  2T\frac{\eta_M}{\pi}\right) \frac{\sqrt{2\eta_M}}{\pi \sqrt{N}}\| \psi(t) \|_{L^\infty((0,T);H^1_0)} \;  \left(\sum_{k=1}^\infty  \frac{1}{k^2\pi^2}\right)^{1/2}.
\end{split}\end{equation*}
Here, we have used the estimate $\lambda_k(t)\geq \nu_k=k^2\pi^2$, which is an easy consequence of the min-max principle for the eigenvalues. This concludes the proof.  \qedhere
\end{proof}

\begin{lemma} \label{le_1}
Under the hypotheses of Theorem \ref{th_conv}, the following estimate holds,
\begin{equation} \label{eq_es1}
\| \psi(t)-\psi_N(t)\|_{L^2}\leq (1 +  2T\eta_M/\pi) \; \| (I-P^N)\psi(t)\|_{L^\infty((0,T);H^1_0)}, \quad t>0 .
\end{equation}
\end{lemma}

\begin{proof}
We apply the projection operator $P^N$ to the equation \eqref{eq_1_bis},
\begin{equation*}
i\partial_t P^N\psi = P^N A_t \psi=P^N A_t P^N\psi+P^N A_t (I-P^N)\psi, \quad P^N\psi(0)=P^N \psi_0.
\end{equation*}
Therefore, $P^N \psi$ satisfies system \eqref{eq_ab1N} with an extra term. If we define $\Gamma_{s,t}^N$ the flow associated to the dynamics of the equation \eqref{eq_ab1N} then,
$$
P^N\psi=\Gamma_{0,t}^N P^N \psi_0 + \int_0^t \Gamma_{s,t}^N P^N A_s (I-P^N)\psi(s)\; ds = \psi_N (t)+ \int_0^t \Gamma_{s,t}^N P^N A_s (I-P^N)\psi(s)\; ds. 
$$
Thus, we obtain 
$$
\psi(t)-\psi_N(t)= (I-P^N) \psi(t) +  \int_0^t \Gamma_{s,t}^N P^N A_s (I-P^N)\psi(s)\; ds.
$$
Thanks to the Minkowski inequality and the fact that $\Gamma_{s,t}^N$ is unitary in $L^2(0,1)$, we have
\begin{equation} \label{eq_conv}
\| \psi(t)-\psi_N(t)\|_{L^2}\leq \|(I-P^N) \psi(t)\|_{L^2} +  \int_0^t \|  P^N A_s (I-P^N)\psi(s)\|_{L^2} \; ds.
\end{equation}
We now estimate the second term of the right hand side of the identity \eqref{eq_conv} by analyzing the operator $P^N A_t (I-P^N)$. For a given $v\in H^1_0 = D(|\Delta|^{1/2})$, we can write 
$$
v(x)=\sum_{j=1}^\infty \hat v_k w_k(x), \quad \quad \hat v_k=\int_0^1 v(x)\overline{w_k(x)} \; dx ,   \quad  \quad  \sum_{k=1}^\infty |\hat v_k|^2 \nu_k <\infty.
$$ 
Thus, $(I-P^N)v=\sum_{j=N+1}^\infty \hat v_j w_j$ and 
$$
A_t(I-P^N)v(x)=\sum_{j=N+1}^\infty \nu_j \hat v_j  w_j(x)+ \eta(t)\delta_{a(t)} (x) \left(\sum_{j=N+1}^\infty \hat v_j w_j(x)\right),
$$
which implies 
$$
P^NA_t(I-P^N)v(x)=P^N \left(\eta(t)\delta_{a(t)} (x) \left(\sum_{j=N+1}^\infty \hat v_j w_j(x)\right)\right)=\eta(t)\sum_{k=1}^N\sum_{j=N+1}^\infty \hat v_j w_j(a(t))w_k(a(t)) w_k(x).
$$
The previous relation yields that
\begin{eqnarray} \nonumber                               
&& \| P^NA_t(I-P^N)v \|_{L^2}^2 = \eta^2(t) \sum_{k=1}^N |w_k(a(t))|^2 \left| \sum_{j=N+1}^\infty \hat v_j w_j(a(t))\right|^2 \\ \nonumber
&& \quad \leq \eta^2(t) 2N \left( \sum_{j=N+1}^\infty \nu_j|\hat v_j|^2 \right) \left( \sum_{l=N+1}^\infty \frac{|w_l(a(t))|^2}{\nu_l}\right)\\ \nonumber
&& \quad \leq \eta^2(t) 2N \|(I-P^N)v\|_{H^1_0}^2 \left( \sum_{j=N+1}^\infty \frac{2}{\nu_j}\right)\\ \label{eq_18}
&& \quad \leq \frac{4}{\pi^2}\eta^2(t) \|(I-P^N)v\|_{H^1_0}^2 N\sum_{j=N+1}^\infty \frac1{j^2}
\leq \frac{4}{\pi^2} \eta^2(t) \|(I-P^N)v\|_{H^1_0}^2.
\end{eqnarray}
In the last computations, we used that $\max_{x\in[0,1]}|w_k(x)|^2\leq 2$ for all $k\geq 1$ and the estimate
$$
N \sum_{j=N+1}^\infty \frac1{j^2} \leq N\int_{N}^\infty \frac1{x^2} dx =1. 
$$ 
Finally, if we apply inequality \eqref{eq_18} in \eqref{eq_conv}, then we obtain \eqref{eq_es1}. This concludes the proof.\qedhere
\end{proof}

\begin{lemma} \label{le_2}
Under the previous hypotheses on $a$ and $\eta$, the following holds, 
\begin{equation} \label{eq_le2}
\| (I-P^N) \phi_k(t,\cdot) \|_{H^1_0} \leq \frac{\sqrt{2}}{\pi}\frac{\sqrt{\eta(t)}}{\sqrt{N}}\sqrt{\lambda_k-\nu_1}.
\end{equation}
\end{lemma}
\begin{proof} We first compute the Fourier coefficients of the eigenfunction $\phi_k$ in the basis $\{ w_j\}_{j\in\N^*}$
\begin{eqnarray*}
c_j^k(t)&=&\int_0^1 \phi_k(t,x) w_j(x) \; dx =- \frac1{\nu_j} \int_0^1 \phi_{k}(t,x) w_{j}''(x)\; dx 
\\
&=& \frac1{\nu_j} \left(-\int_0^1 \phi_{k}(x,t) w_{j}''(x)\; dx+ \eta(t)\phi_{k}(a(t),t) w_{j}(a(t)) \right) - \frac1{\nu_j}\eta(t)\phi_{k}(a(t),t) w_{j}(a(t))\\
&=&\frac{\lambda_k}{\nu_j} \int_0^1 \phi_k(x,t) w_j(x) \; dx - \frac1{\nu_j}\eta(t)\phi_{k}(a(t),t) w_{j}(a(t))\\
&=&\frac{\lambda_k}{\nu_j} c_j^k- \frac1{\nu_j}\eta(t)\phi_{k}(a(t),t) w_{j}(a(t)).
\end{eqnarray*}
Therefore,
\begin{equation} \label{eq_int1}
|c_j^k|\leq \frac{\eta(t)}{|\nu_j-\lambda_k|} |\phi_k(t,a(t))w_j(a(t))|\leq \sqrt{2}\frac{\eta(t)}{|\nu_j-\lambda_k|} |\phi_k(t,a(t))|. 
\end{equation}
On the other hand, from the $L^2-$normalization of the eigenfunctions, we have
$$
\lambda_k=(A_t\phi_k,\phi_k)=\int_{0}^1 |\phi_{k}'(t,x)|^2 dx+\eta(t)|\phi_{k}(t,a(t))|^2. 
$$
Therefore, by Poincare inequality
$$
\eta(t)|\phi_{k}(t,a(t))|^2\leq \lambda_k-\nu_1 .
$$
Substituting in \eqref{eq_int1}, we finally obtain
$$
|c_j^k|\leq  \sqrt{2}\frac{\sqrt{\eta(t)}}{\nu_j-\lambda_k} \sqrt{\lambda_k-\nu_1}.
$$
The relation \eqref{eq_le2} is obtained thanks to the last inequality since
\begin{equation*}\begin{split}
\| (I-P^N) \phi_k \|_{H^1_0}^2 &=\sum_{j=N+1}^\infty |c_j^k|^2 \nu_j \leq 2\eta(t)(\lambda_k-\nu_1) \sum_{j=N+1}^\infty \frac{\nu_j}{(\nu_j-\lambda_k)^2} \\ &\leq 2\eta(t)(\lambda_k-\nu_1) \sum_{j=N+1}^\infty \frac{1}{(\nu_j-1)^2} \leq  2\eta(t)(\lambda_k-\nu_1) \frac1{\pi^2N}.\qedhere
\end{split}\end{equation*}\end{proof}

\subsection{Matrix formulation of the Galerkin method}\label{matrixform}
The aim of this subsection is to rewrite the finite dimensional system \eqref{eq_ab1N} in matrix form. Let $\psi_N\in C^1([0,T],X_N)$ be a solution of \eqref{eq_ab1N}. We consider its decomposition
$$
\psi_N(t,x)= \sum_{n= 1}^N \hat \psi_n(t) \sqrt{2} \sin(n\pi x).
$$
If we substitute this expression in \eqref{eq_ab1N} and project on the subspace generated by the eigenfunction $\sqrt{2} \sin(n\pi x)$ with $n=1,...,N,$ then we obtain
$$
i\hat \psi_n'(t)=n^2\pi^2 \hat \psi_n(t) +2\eta(t)\sum_{k= 1}^N \hat \psi_k(t)\sin(k\pi a(t) )\sin(n\pi a(t)), \quad n=1,...,N.
$$
We write this relation in matrix form as follows
$$
\Psi'=(D+P)\Psi,
$$
where $P(t)=-2i\eta(t) \; s\otimes s$ for  $s=(\sin(j\pi a(t)))_{j=1}^N$ and
$$
\Psi=\left( \begin{array}{c} \hat \psi_1\\ \hat\psi_2\\ ...\\ \hat \psi_N \end{array}\right), \quad D=-i \pi^2 \left( \begin{array}{cccc} 1^2 &0 &...&0 \\ 0&2^2&...&0\\ ...\\ 0&0&...&N^2 \end{array}\right).$$
The initial condition is obtained by projecting the Fourier representation of $\psi_0(x)$ as follows
$$
P^N\psi^0(x)=\sum_{n= 1}^N \hat \psi^0_n \sqrt{2} \sin(n\pi x).
$$
Therefore
$$
\Psi(0)=\Psi^0=\left( \begin{array}{c} \hat \psi_1^0\\ \hat \psi_2^0\\ ...\\ \hat \psi_N^0 \end{array}\right).
$$
Finally, we obtain the first order system 
\begin{equation} \label{eq_time}
\left\{ 
\begin{array}{l}
\Psi'=(D+P(t))\Psi,\\
\Psi(0)=\Psi^0.
\end{array}
\right.
\end{equation}

\subsection{Time discretization}

To solve system (\ref{eq_time}), we propose a midpoint implicit method in time:
\begin{equation}\label{eq_trpz}
\frac{\Psi^{k+1} -\Psi^k}{\Delta t} = \left(D+P\Big(\frac{t_{k+1}+t_k}{2} \Big)\right)\frac{\Psi^{k+1} +\Psi^k}{2}, \quad k=0,...,K-1, 
\end{equation}
that is second order accurate and convergent as long as 
$\|P''(t)\|$ is bounded in $t\in [0,T]$. Another advantage of this method, especially for long time simulations, is that it conserves the $L^2$-norm since $D$ and $P$ are purely imaginary matrixes as stated in the following result.
\begin{lemma}
Let $K\in\N^*$, $\Delta t=T/K$ and $t_k=k\; \Delta t $ with $k=0,...,K$ be a uniform mesh of the interval $[0,T]$. Consider $\{\Psi^k\}_{k=0}^K$ the discrete solution of (\ref{eq_time}) obtained with the implicit midpoint method (\ref{eq_trpz}).
The following identity is verified
\begin{equation} \label{eq_cons}
\| \Psi^{K} \|^2 = \| \Psi^{0} \|^2 . 
\end{equation}    
\end{lemma}   

\begin{proof}
Multiplying the system (\ref{eq_trpz}) by $(\overline{\Psi^{k+1}}+\overline{\Psi^k})/2$ and taking the real part, we easily obtain
$$
\| \Psi^{k+1} \|^2 - \| \Psi^{k} \|^2 = 0. 
$$
The relation  (\ref{eq_cons}) is ensured by summing the previous equation in $k=0,...,K-1$.
\end{proof}
\section{Numerical experiments}\label{numerical}

In this section, we give some implementation details and provide some numerical experiments illustrating the performance of the numerical method to show the energy permutation described in the introduction and ensured in Section \ref{permutation}. 

We are interested in capturing permutation of energy between different modes with a very particular type of controls involving delta measures located in some moving points. The whole process can be divided in several time periods that change at some intermediate times.

Let us first comment the difficulty to estimate both analytical and discretization parameters for such experiments. To fix ideas, we focus on the example given in the introduction where the permutation between the first two modes is described in 5 consecutive time intervals with a single Dirac potential. Concerning the analytical parameters, we have to choose suitable values for $\eta_M$ and $T$. On the one hand, $\eta_M$ must be sufficiently large to mimic the dynamics with the split interval $(0,a(t))\cup (a(t),1)$ as done in \cite{mio7}. In our experiments, this is almost the case for $\eta_M\sim 2000$. On the other hand, we have to choose $T$ large enough in order to guarantee adiabatic regimes both when we include/remove the Diracs (time subintervals $[0,T_1]$ and $[T_4,T]$) and when we move their supports (time subintervals $[T_1,T_2]$ and $[T_3,T_4]$). However, common results on adiabatic dynamics do not directly provide explicit estimates on how slow $\eta(t)$ and $a(t)$ should be to obtain such regime. Of course, this affects the final time $T$ which must be estimated from numerical experiments.

Concerning the discretization parameters, we have to choose $N$ and the time step $\Delta t$. According to (\ref{eq_th1}), error estimates depend on $TN^{-1/2}\eta_M^{3/2}$, up to a constant. Taking into account $N<200$ (to limit the computational cost) and $\eta_M=2000$, we see that this error estimate is not very useful even for $T=1$.

Fortunately, things work much better in practice. If we consider the control process described in the introduction, then we found that, for $N=200$, $\eta_M=2000$ and 
$$  
\eta(t)=\left\{ 
\begin{array}{ll}
\eta_M (1-\cos(\pi t / T_1))/2, & t\in [0,T_1],\\
\eta_M, & t\in [T_1,T_2],\\
\eta_M (1-\cos(\pi (t-T_2) / (T-T_2)))/2, & t\in [T_2,T],
\end{array}\right.
$$
an adiabatic regime is obtained both in the interval $[0,T_1]$ and $[T_4,T]$ as long as $T_1,T-T_4>50$, with maximum relative error of $10^{-2}$ in the energy distribution. 

Concerning the movement of $a(t)$, we show in Figure \ref{fig1_a} how a permutation of the energy is performed for a linear transition when $a'\sim 1$, while the dynamics is adiabatic for $a'\sim 10^{-4}$. However, the behavior in the intermediate region may strongly depend on the initial distribution of the energies. Thus, simulations taking into account velocities $a'$ in this intermediate region are very challenging. Note also that the transition is not monotone and there is a value of $a'$ for which the energy is transferred to one of the modes.    

\begin{figure}[h]
\begin{center}
\begin{tabular}{cc}
\includegraphics[height=5cm]{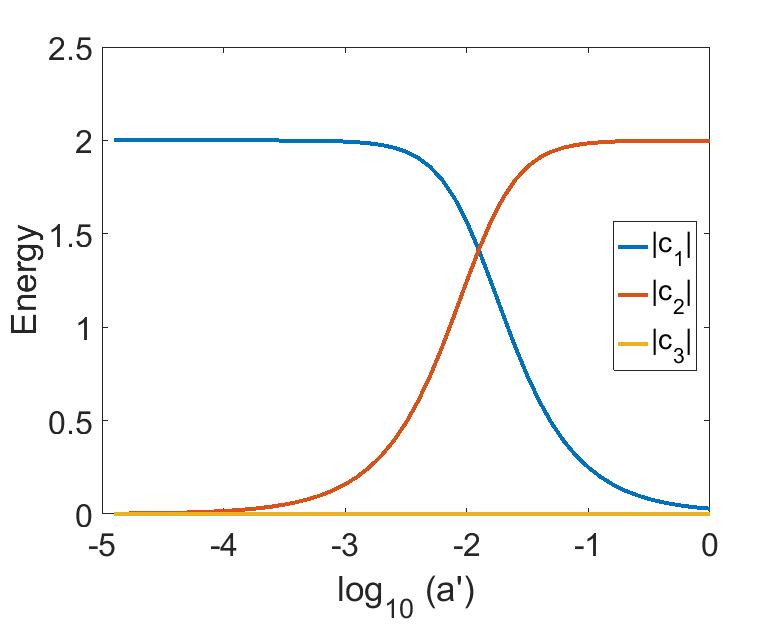} &
\includegraphics[height=5cm]{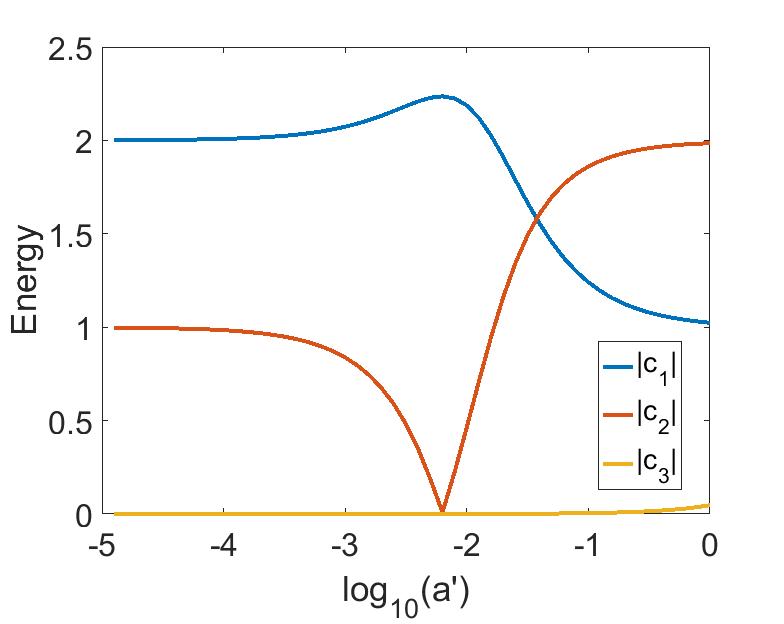} \\
$\psi_0(x)=2\phi_1(x,0)$ & $\psi_0(x)=2\phi_1(x,0)+\phi_2(x,0)$ 
\end{tabular}
\end{center}
\caption{The figure represents the distribution of the energy of the first $3$ modes after $T=0.01$. The control considered in the dynamics is $V(x,t)=\eta_M\delta_{x=a(t)}$ with $\eta_M=2000$ and
$x=a(t)$ a linear trajectory traversing $x=0.5$, from the right to the left, with different slopes $a'$. As indicated, the initial energy distribution ($t=0$) is different in both simulations.
We see that for $a'\sim 1$, there is a permutation of the energy while this permutation is lost as the slope decreases. For $a'\sim 10^{-4}$, the dynamics can be considered as adiabatic. We also observe that the transition region is very sensitive to the values of the energies we are permuting. \label{fig1_a}}
\end{figure}

In the experiments below, we set the parameters as follows.
\begin{enumerate}
\item  Take $N=200$, time step $\Delta t=10^{-3}$, $\eta_M\sim 2000$ and the time intervals where $\eta$ increases and decreases of length larger than $50$.
\item When we need to move $a(t)$ adiabatically, we consider $a'\sim 10^{-4}$.
\item When a non-adiabatic regime is required, we set $a'=1$ in a time interval of length $10^{-2}$.
\end{enumerate} 

It is important to note that with this choice, we are able to simulate both the adiabatic regime and the continuous non-adiabatic transition required in the control strategy. However, the phase in the numerical approximation can be very different from the one in the real solution.   

\subsection{Experiment 1} 
In this experiment, we apply the control strategy introduced in Section \ref{sec_cont_stra} to permute the energy of the first 3 Fourier modes. We follow the notation introduced there. We start with an initial data of the form 
\begin{equation} \label{exp1_T0}
\psi_0(x)=\psi(x,0)=\sum_{i=1}^3c_i(0)\phi_i(0,x),
\end{equation}
and we want obtain at time $t=T$
\begin{equation} \label{exp1_T_a}
\psi(T,x)=\sum_{i=1}^3c_i(T)\phi_i(0,x),
\end{equation}
with
$$
|c_1(T)|=|c_2(0)|, \qquad |c_2(T)|=|c_3(0)|, \qquad |c_3(T)|=|c_1(0)|.
$$
Note that this corresponds with the permutation $(\sigma(1),\sigma(2),\sigma(3))=(2,3,1)$.
We choose $c_1(0)=1$, $c_2(0)=1.5$, $c_3(0)=2$.
The strategy is then to introduce two Dirac potentials at suitable points and combine the adiabatic regime with some fast nonadiabatic transitions.  Time should be large to ensure the adiabatic process, both when we increase the Dirac masses and when we move them. We consider $T=200$. The weight of the Dirac $\eta(t)$ grows in $t\in [0,T/4]$ from 0 to $\eta_M$ and decreases from $\eta_M$ to 0 in $t\in[3T/4,T]$.

 The two Diracs are introduced in the first time interval $[0,T/4]$ adiabatically, and allow us to localize the three modes in three consecutive subintervals $I_j^s\subset(0,1)$, $j=1,...,3$, of decreasing lengths i.e. $l_1^s>l_2^s>l_3^s>l_1^s/2$. The left subinterval is the largest and the first mode is localized here. The middle one contains the second mode and the right one the third mode  (see Figure \ref{fig1}). According to the prescribed permutation, we move the Diracs in such a way that the lengths of the final subintervals satisfy $l_2^f>l_3^f>l_1^f >l_2^f/2$, maintaining the structure of the solution. In order to do that, we combine slow adiabatic movements with fast nonadiabatic transitions, whenever the lengths of two subintervals become equal. Thus, at time $t=3T/4$, we have a solution with the same aspect as when $t=T/4$, but the intervals $I_j^f\subset(0,1)$ have different lengths. In the final time subinterval $[3T/4,T]$, we remove adiabatically the two Diracs.  

In this particular example, we introduce the Diracs at the points $b_1^s=0.36$ and $b_2^s=0.7$ so that $l_1^s=0.36>l_2^s=0.34>l_3^s=0.3>l_1^s/2=0.18$.  Now, we fix $b_2^s$ and move $b_1^s$ from $0.36$ to $b_1^f=0.31$ with a transition at $x=0.35$, where the lengths of the first two intervals coincide. Then, we fix $b_1^f=0.31$ and move $b_2^s$ from $0.7$ to $b_2^f=0.68$ with a transition at $x=0.69$, where the first and third intervals have the same length. After these movements, we have $l_2^f=0.37>l_3^f=0.32>l_1^f=0.31>l_2^f/2=0.185$. In the intermediate region between the transitions, we move the Diracs adiabatically. 
The trajectories $a_1(t)$, $a_2(t)$ are also given in Figure \ref{fig1}.

\begin{figure}[H]
\begin{center}
\begin{tabular}{cc}
\includegraphics[height=5cm]{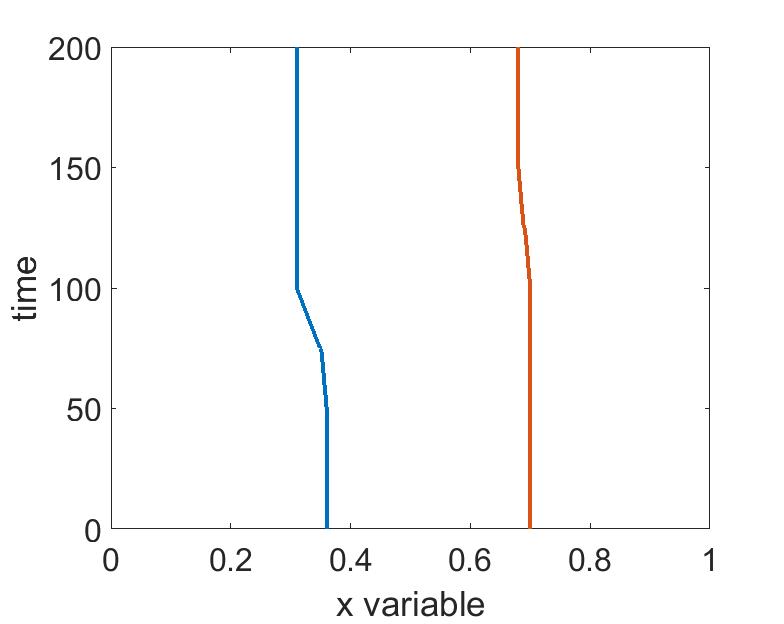} & \includegraphics[height=5cm]{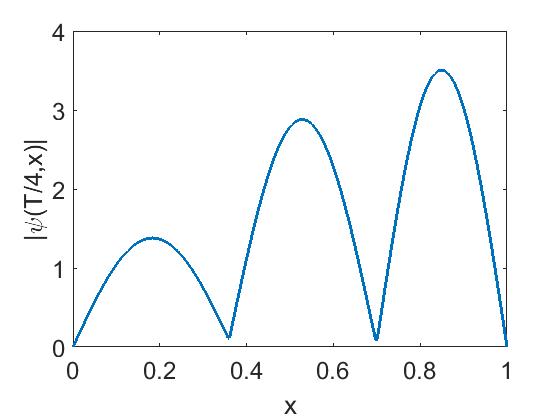} \\
Trajectories of the Diracs & Solution after including the Diracs.
\end{tabular}
\end{center}
\caption{Experiment I: The figure represents the trajectories $a_1(t)$ and $a_2(t)$, and the modulus of the solution after including the Diracs. In the positions of the Diracs, the solution is almost zero and the dynamics is similar to the one associated to a split domain. We observe how the different modes are localized in the three subintervals between the Diracs. \label{fig1}}
\end{figure}

In Figure \ref{fig1_b}, we show the behavior of the first 3 modes along the time (left plot) and the energy associated to the first Fourier coefficients (right plot). We see how this energy is shifted at the two jumps of the trajectories of $a_1(t)$ and $a_2(t)$. 

\begin{figure}[H]
\begin{center}
\begin{tabular}{cc}
\includegraphics[height=5cm]{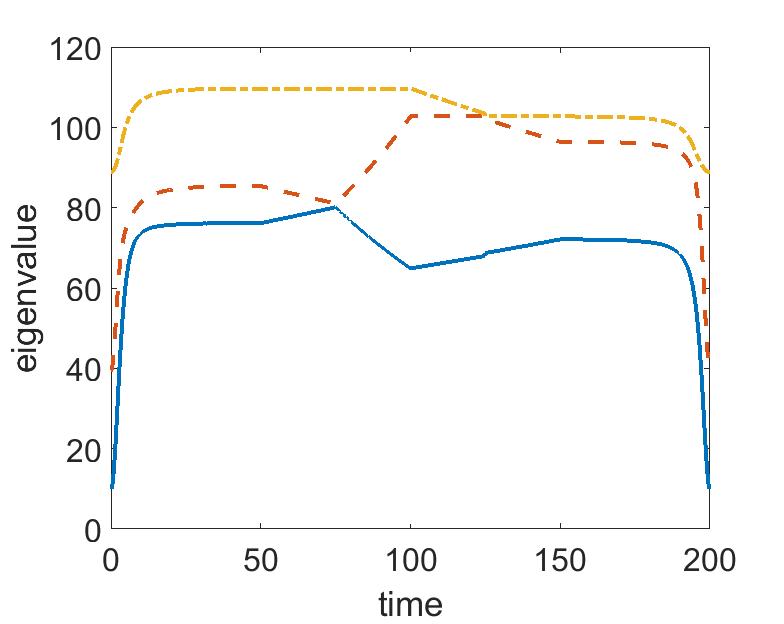} & \includegraphics[height=5cm]{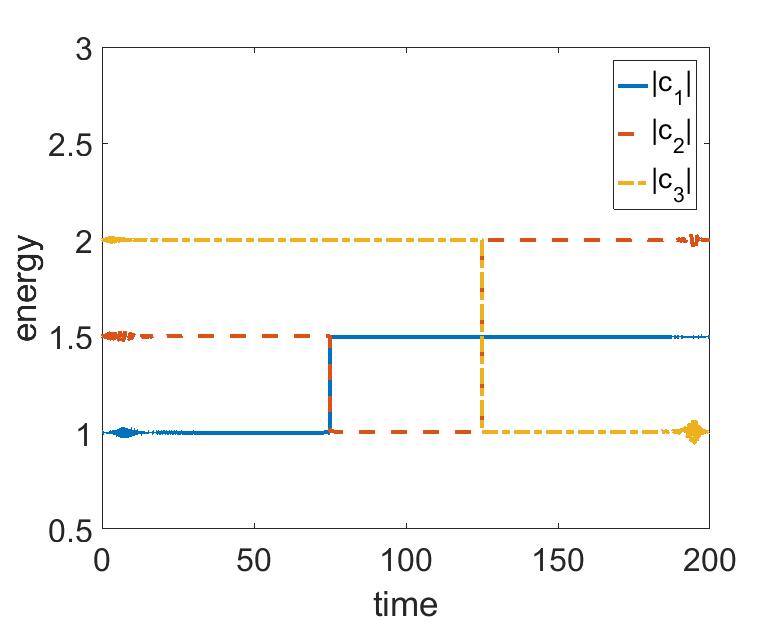} 
\end{tabular}
\end{center}
\caption{Experiment I: The figure represents the time evolution of the first 3 eigenvalues (left plot) and of the energy associated to the first Fourier coefficients (right plot). \label{fig1_b}}
\end{figure}

In our experiment, we do not obtain exactly \eqref{exp1_T_a}. The initial data is given by \eqref{exp1_T0}. The solution at time $T=200$ is given by 
\begin{equation} \label{exp1_T}
\psi(T,x)=\sum_{i=1}^3 c_i(T)w_i(x) + w(x),
\end{equation}
where $\|w\|_{L^2}= 1.1\times 10^{-2}$ and
$$
\big||c_3(T)|-|c_1(0)|\big|=1.2\times 10^{-3}, \qquad \big||c_2(T)|-|c_3(0)|\big|=1.12\times 10^{-4},\qquad \big||c_1(T)|-|c_2(0)|\big|=4.01\times 10^{-3}.$$

\subsection{Experiment 2}
In this experiment, we follow the control strategy to change the energy of the first 4 Fourier modes according to the permutation $(\sigma(1),\sigma(2),\sigma(3), \sigma(4))=(2,4,3,1)$. Thus, we start with an initial data of the form 
\begin{equation} \label{exp2_T0}
\psi_0(x)=\psi(x,0)=\sum_{i=1}^4c_i(0)w_i(x),
\end{equation}
and we want to obtain at time $t=T$
\begin{equation} \label{exp2_T_a}
\psi(T,x)=\sum_{i=1}^4c_i(T)w_i(x),
\end{equation}
with
$$
|c_1(T)|=|c_4(0)|, \qquad |c_2(T)|=|c_1(0)|, \qquad |c_3(T)|=|c_3(0)|, \qquad |c_4(T)|=|c_2(0)|.
$$
We choose $c_1=0.5$, $c_2=1$, $c_3=1.5$, $c_4=2$ and $T=200$. 

As in the Experiment I, the three Diracs are introduced in the first time interval $[0,T/4]$ adiabatically to localize the four modes in consecutive subintervals of decreasing length $l_j^s\subset(0,1)$, $j=1,...,4$. The left subinterval $l_1^s$ is the largest and it contains the first mode, while the fourth mode is localized in the right one $l_4^s$. 
Then, we move the Diracs changing the lengths of these subintervals to obtain the prescribed permutation i.e. $l_2^f>l_4^f>l_3^f>l_1^f>l_2^f/2$, but maintaining the structure of the solution.   

In this example, we introduce the Diracs at the points $b_1^s=0.27$, $b_2^s=0.53$ and $b_3^s=0.77$ so that $l_1^s=0.27>l_2^s=0.26>l_3^s=0.24>l_4^s=0.23>l_1^s/2=0.135$.  Afterwards, we move the Diracs. Firstly, we fix $b_2^s, b_3^s$ and move $b_1^s$ from $0.27$ to $b_1^f=0.26$ with a transition at $x=0.265$ in the time interval $t\in[T/4,T/4+T/8]$. Secondly, we fix $b_1^f=0.26$, $b_3^s=0.77$ and move $b_2^s$ from $0.53$ to $\tilde b_2^f=0.5$ with three transitions at $x=0.52,$ $x=0.515$ and $x=0.51$, in the time interval $t\in[T/4+T/8,T/2]$. Thirdly, we fix $b_1^f=0.26$, $\tilde b_2^f=0.5$ and move $b_3^s$ from $0.77$ to $b_3^f=0.73$ with three transitions at $x=0.76$, $x=0.75$ and $x=0.74$, in the time interval $t\in[T/2,T/2+T/8]$. Finally, we fix $b_1^f=0.26$, $b_3^f=0.73$ and move again $\tilde b_2^f$ from $0.5$ to $b_2^f=0.49$ with a transition at  $x=0.495$ in the time interval $t\in [T/2+T/8,3T/4]$. Overall, we perform $8$ energy transitions at different times.
After these movements, we have $l_4^f=0.27>l_1^f=0.26>l_3^f=0.24>l_2^f=0.23>l_4^f/2=0.135$. The trajectories $a_1(t)$, $a_2(t)$, $a_3(t)$ are given in Figure \ref{fig3}.

\begin{figure}[h]
\begin{center}
\begin{tabular}{cc}
\includegraphics[height=4.5cm]{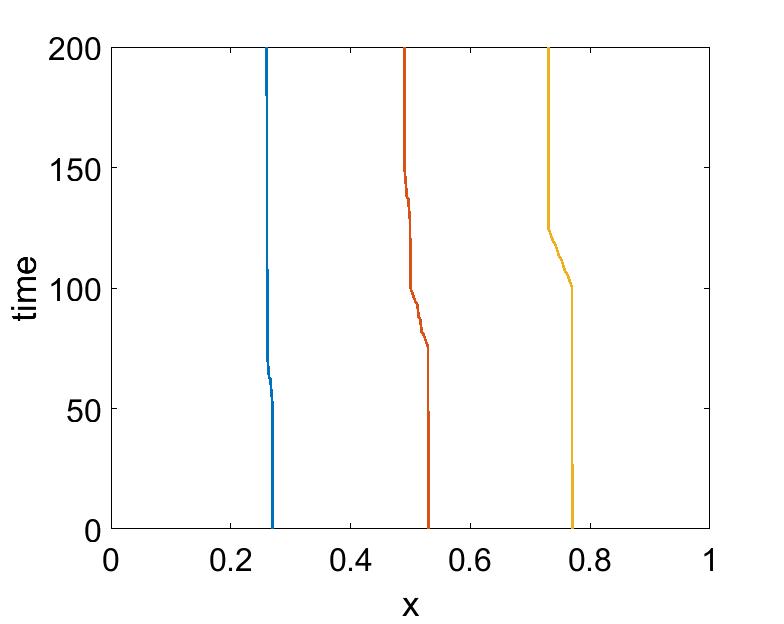} & \includegraphics[height=4.5cm]{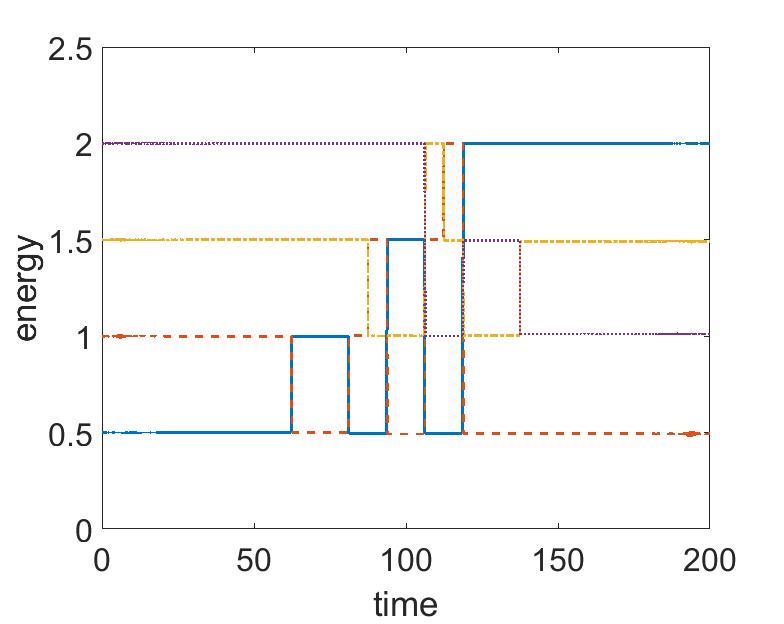} \\
Trajectories of the Diracs & energy jumps of the Fourier modes 
\end{tabular}
\end{center}
\caption{Experiment II: The figure represents the trajectories $a_1(t)$, $a_2(t)$ and $a_3$, and the energy jumps appearing during the dynamics \label{fig3}}
\end{figure}

In our experiment, we do not obtain exactly \eqref{exp2_T_a}. The initial data is given by \eqref{exp2_T0}. The solution at  time $T=200$ is given by 
\begin{equation} \label{exp2_T}
\psi(T,x)=\sum_{i=1}^4 c_i(T)w_i(x) + w(x),
\end{equation}
where $\|w\|_{L^2}= 5.76\times 10^{-2}$ and
\begin{eqnarray*}
\big||c_1(T)|-|c_4(0)|\big|&=&8.8\times 10^{-3}, \qquad \big||c_2(T)|-|c_1(0)|\big|=1.95\times 10^{-2},\\ 
\big||c_3(T)|-|c_3(0)|\big|&=&9.80\times 10^{-3},
\qquad \big||c_4(T)|-|c_2(0)|\big|=1.80\times 10^{-2}.
\end{eqnarray*}

\subsection{Comments}\label{comments}

As described in \cite{mio7}, from any initial data, it is possible to achieve not only a permutation but a complete redistribution between the energy modes with the type of controls described in this paper. This is done in the two following steps:
\begin{enumerate}
\item collapse of the energy associated to all the modes in the initial data to a single mode;
\item share of the energy associated to this single mode in order to obtain the desired final state.   
\end{enumerate} 
This strategy is valid in our model but the numerical approximation introduce an important difficulty since it requires to approximate an intermediate transition between the adiabatic and non-adiabatic dynamics. This approximation is very unstable and strongly depends on the energies we are considering, as we commented above in Figure \ref{fig1_a}. In addition, it is also very sensitive to the discretization parameters $N$ and the time step. This makes very difficult to design {\em a priori} the path for the support of the Dirac achieving this intermediate transition.  

For the permutation described in the previous experiments, the correct path combining the adiabatic regime and the transitions can be designed {\em a priori} following a simple strategy, as we have described. This provides a very explicit control which is easily approximated.  

\section*{Acknowledgements}

The first author acknowledges the support of Ministerio de Ciencia, Innovaci\'on y Universidades of
the Spanish government through the grant MTM2017-85934-C3-3-P. The second author was financially supported by the Agence nationale de la recherche of the French government through the grant {\it 
ISDEEC} (ANR-16-CE40-0013).

\end{document}